\documentclass[a4,11pt]{article}
\usepackage[subrefformat=parens]{subcaption}
\usepackage[dvips]{graphicx}
\usepackage{amssymb}
\usepackage{amstext}
\usepackage{amsmath}
\usepackage{amscd}
\usepackage{amsthm}
\usepackage{amsfonts}
\usepackage{enumerate}
\usepackage{graphicx}
\usepackage{latexsym}
\usepackage{mathrsfs}
\usepackage{mathtools}
\usepackage{cases}
\usepackage[svgnames]{xcolor}
\usepackage{verbatim}
\usepackage{bm}
\usepackage{tikz}
\newtheorem{thm}{Theorem}[section]
\newtheorem{cor}[thm]{Corollary}
\newtheorem{lem}[thm]{Lemma}

\newtheorem{ass}[thm]{Assumption}
\theoremstyle{definition}
\newtheorem{rem}[thm]{Remark}
\newtheorem{dfn}[thm]{Definition}

\newtheorem*{claim*}{Claim}

\numberwithin{equation}{section}

\title{\vspace{-3cm}\textbf{The factorization and monotonicity method for the defect in an open periodic waveguide}}

\author{Takashi FURUYA}

\date{}
\begin{document}
\maketitle
\begin{abstract}
In this paper, we consider the inverse scattering problem to reconstruct the defect in an open periodic waveguide from near field data. Our first aim is to mention that there is a mistake in the factorization method of Lechleiter \cite{Lechleiter1}. By this we can not apply it to solve this inverse problem. Our second aim is to give ways to understand the defect from inside (Theorem 1.1) and outside (Theorem 1.2) based on the monotonicity method. Finally, we give numerical examples based on Theorem 1.1.
\end{abstract}
\section{Introduction}
In this paper,  we consider the inverse scattering problem to reconstruct the defect in an open periodic waveguide from near field data. The contributions of this paper are followings.
\begin{itemize}
  \item We mention that there is a mistake in ($F_{\#}$--) factorization method of Lechleiter \cite{Lechleiter1} by giving a counterexample, and revise the proof of the factorization method.
  \item We give the reconstruction method based on the monotonicity method. The unknown defect are understood from inside (Theorem 1.1) and outside (Theorem 1.2)
\end{itemize}
By the mistake we will understand that {\it transmission eigenvale} is needed in the factorization method for inverse medium scattering problem. Furthermore, we can not apply it to inverse problem of an open periodic waveguide we will consider here. In order to solve the inverse problem of our case, we consider the monotonicity method. The feature of this method is to understand the inclusion relation of an unknown defect and an artificial domain by comparing the data operator with some operator corresponding to an artificial one. For recent works of the monotonicity method, we refer to \cite{Furuya2, R. Griesmaier and B. Harrach, B. Harrach and V. Pohjola and M. Salo1, B. Harrach and V. Pohjola and M. Salo2, B. Harrach and M. Ullrich1, B. Harrach and M. Ullrich2, Lakshtanov and Lechleiter}. 
\par
We begin with formulation of the scattering problem. Let $k>0$ be the wave number, and let $\mathbb{R}^2_{+}:=\mathbb{R}\times (0, \infty)$ be the upper half plane, and let $W:=\mathbb{R}\times (0, h)$ be the waveguide in $\mathbb{R}^2_{+}$. We denote by $\Gamma_a:=\mathbb{R}\times\{ a\}$ for $a>0$. Let $n \in L^{\infty}(\mathbb{R}^2_{+})$ be real value, $2\pi$-periodic with respect to $x_1$ (that is, $n(x_1+2\pi,x_2)=n(x_1,x_2 )$ for all $x=(x_1,x_2) \in \mathbb{R}^2_{+}$), and equal to one for $x_2>h$. We assume that there exists a constant $n_{max}>0$ and $n_{min}>0$ such that $n_{min} \leq n \leq n_{max}$ in $\mathbb{R}^2_{+}$. Let $q \in L^{\infty}(\mathbb{R}^2_{+})$ be real value with the compact support in $W$. We denote by $Q:=\mathrm{supp}q$. Assume that $\mathbb{R}^2 \setminus \overline{Q}$ is connected. First of all we consider the following direct scattering problem: For fixed $y \in \mathbb{R}^2_{+} \setminus \overline{W}$, determine the scattered field $u^{s} \in H^{1}_{loc}(\mathbb{R}^2_{+})$ such that
\begin{equation}
\Delta u^{s}+k^2(1+q)nu^{s}=-k^{2}qnu^{i}(\cdot, y) \ \mathrm{in} \ \mathbb{R}^2_{+}, \label{1.1}
\end{equation}
\begin{equation}
u^{s}=0 \ \mathrm{on} \ \Gamma_0, \label{1.2}
\end{equation}
where the incident field $u^{i}$ is given by $u^{i}(x,y)=\overline{G_n(x,y)}$, where $G_n$ is the Dirichlet Green's function in the upper half plane $\mathbb{R}^2_{+}$ for $\Delta +k^2n$, that is,
\begin{equation}
G_n(x,y):=G(x,y)+\tilde{u}^{s}(x,y), \label{1.3}
\end{equation}
where $G(x,y):=\Phi_k(x,y)-\Phi_k(x,y^{*})$ is the Dirichlet Green's function for $\Delta +k^2$, and $y^{*}=(y_1, -y_2)$ is the reflected point of $y$ at $\mathbb{R}\times \{0\}$. Here, $\Phi_k(x,y)$ is the fundamental solution to Helmholtz equation in $\mathbb{R}^2$, that is, 
\begin{equation}
\Phi_k(x,y):= \displaystyle \frac{i}{4}H^{(1)}_0(k|x-y|), \ x \neq y. \label{1.4}
\end{equation}
$\tilde{u}^{s}$ is the scattered field of the unperturbed problem by the incident field $G(x,y)$, that is, $\tilde{u}^{s}$ vanishes for $x_2=0$ and solves
\begin{equation}
\Delta \tilde{u}^{s}+k^2n\tilde{u}^{s}=k^{2}(1-n)G(\cdot, y) \ \mathrm{in} \ \mathbb{R}^2_{+}. \label{1.5}
\end{equation}
If we impose a suitable radiation condition introduced by Kirsch and Lechleiter \cite{Kirsch and Lechleiter2}, the unperturbed solution $\tilde{u}^{s}$ is uniquely determined. Later, we will explain the exact definition of this radiation condition (see Definition 2.4). Furthermore, with the same radiation condition and an additional assumption (see Assumption 2.7) the well-posedness of the problem (\ref{1.1})--(\ref{1.2}) was show in \cite{Furuya1}.
\par
By the well-posedness of this perturbed scattering problem, we are able to consider the inverse problem of determining the support of $q$ from measured scattered field $u^s$ by the incident field $u^{i}$. Let $M:=\{(x_1, m): a<x_1<b \}$
for $a<b$ and $m>h$, and $Q:=\mathrm{supp}q$. With the scattered field $u^{s}$, we define the {\it near field operator} $N:L^{2}(M)\to L^{2}(M)$ by
\begin{equation}
Ng(x):=\int_{M}u^{s}(x,y)g(y)ds(y), \ x \in M. \label{1.6}
\end{equation}
The inverse problem we consider in this paper is to determine support $Q$ of $q$ from the scattered field $u^{s}(x,y)$ for all $x$ and $y$ in $M$ with one $k>0$. In other words, given the near field operator $N$, determine $Q$.
\par
Our aim in this paper is to revise the factorization method (see section 3) and to provide the following two theorems.
\begin{thm}
Let $B \subset \mathbb{R}^2$ be a bounded open set. Let Assumption hold, and assume that there exists $q_{min}>0$ such that $q\geq q_{min}$ a.e. in $Q$. Then for $0<\alpha<k^2n_{min}q_{min}$, 
\begin{equation}
B \subset Q \ \ \ \  \Longleftrightarrow \ \ \ \  \alpha H^{*}_{B}H_{B}\leq_{\mathrm{fin}} \mathrm{Re}N,\label{1.7}
\end{equation}
where the operator $H_{B}:L^{2}(M) \to L^{2}(B)$ is given by
\begin{equation}
H_{B}g(x):=\int_{M}\overline{G_n(x,y)}g(y)ds(y), \ x \in B, \label{1.8}
\end{equation}
and the inequality on the right hand side in (\ref{1.7}) denotes that $\mathrm{Re}N-\alpha H^{*}_{B}H_{B}$ has only finitely many negative eigenvalues, and the real part of an operator $A$ is self-adjoint operators given by $\mathrm{Re}(A):=\displaystyle \frac{1}{2}(A+A^{*})$.
\end{thm}
\begin{thm}
Let $B \subset \mathbb{R}^2$ be a bounded open set. Let Assumption hold, and assume that there exists $q_{min}>0$ and $q_{max}>0$ such that $q_{min}\leq q \leq q_{max}$ a.e. in $Q$. Then for $\alpha>k^2n_{max}q_{max}$,
\begin{equation}
Q \subset B \ \ \ \  \Longleftrightarrow \ \ \ \ \mathrm{Re}N \leq_{\mathrm{fin}} \alpha H^{*}_{B}H_{B}, \label{1.9}
\end{equation}
\end{thm}
We understand whether an artificial domain $B$ is contained in $Q$ or not in Theorem 1.1, and $B$ contain $Q$ in Theorem 1.2, respectively. Then, by preparing a lot of known domain $B$ and for each $B$ checking (\ref{1.7}) or (\ref{1.9}) we can reconstruct the shape and location of unknown $Q$.
\par
This paper is organized as follows. In Section 2, we recall a radiation condition introduced in \cite{Kirsch and Lechleiter2}, and the well-posedness of the problem (\ref{1.1})--(\ref{1.2}). In Section 3, we mention the exact functional analytic theorem in the factorization method (Theorem 2.15 in \cite{Kirsch and Grinberg}), and mention where there is a mistake in one of Lechleiter (Theorem 2.1 in \cite{Lechleiter1}) by giving an counterexample. In Section 4, we consider several factorization of the near field operator $N$, and we mention that there is a difficulty to apply the factorization method due to the mistake of Lechleiter. However, the properties of its factorization discussed in Section 4 will be useful when we show Theorems 1.1 and 1.2. In Sections 5 and 6, we prove Theorems 1.1 and 1.2, respectively. Finally in Section 7, numerical examples based on Theorem 1.1 are given.
\section{A radiation condition}
In Section 2, we recall a radiation condition introduced in \cite{Kirsch and Lechleiter2}. Let $f \in L^{2}(\mathbb{R}^2_{+})$ have the compact support in $W$.  First, we consider the following direct problem: Determine the scattered field $u \in H^{1}_{loc}(\mathbb{R}^2_{+})$ such that
\begin{equation}
\Delta u+k^2nu=f \ \mathrm{in} \ \mathbb{R}^2_{+}, \label{2.1}
\end{equation}
\begin{equation}
u=0 \ \mathrm{on} \ \Gamma_0. \label{2.2}
\end{equation}
(\ref{2.1}) is understood in the variational sense, that is,
\begin{equation}
\int_{\mathbb{R}^2_{+}} \bigl[ \nabla u \cdot \nabla \overline{\varphi}-k^2nu\overline{\varphi} \bigr]dx=-\int_W f \overline{\varphi}dx, \label{2.3}
\end{equation}
for all $\varphi \in H^{1}(\mathbb{R}^2_{+})$, with compact support. In such a problem, it is natural to impose the {\it upward propagating radiation condition}, that is, $u(\cdot, h) \in L^{\infty}(\mathbb{R})$ and 
\begin{equation}
u(x)=2\int_{\Gamma_h}u(y)\frac{\partial\Phi_k(x,y)}{\partial y_2} ds(y)=0,\ x_2>h. \label{2.4}
\end{equation}
However, even with this condition we can not expect the uniqueness of this problem. (see Example 2.3 of \cite{Kirsch and Lechleiter2}.) In order to introduce a {\it suitable radiation condition}, Kirsch and Lechleiter discussed limiting absorption solution of this problem, that is, the limit of the solution $u_{\epsilon}$ of $\Delta u_{\epsilon}+(k+i\epsilon)^2nu_{\epsilon}=f$ as $\epsilon \to 0$. For the details of an introduction of this radiation condition, we refer to \cite{Kirsch and Lechleiter1, Kirsch and Lechleiter2}.
\par
Let us prepare for the exact definition of the radiation condition. We denote by $C_{R}:=(0,2\pi) \times (0, R)$ for $R \in (0,\infty]$. The function $u\in H^{1}(C_R)$ is called $\alpha$-quasi periodic if $u(2\pi,x_2)=e^{2\pi i\alpha}u(0,x_2)$. We denote by $H^{1}_{\alpha}(C_R)$ the subspace of the $\alpha$-quasi periodic function in $H^{1}(C_R)$, and $H^{1}_{\alpha, loc}(C_{\infty}):=\{u \in H^{1}_{loc}(C_{\infty}) : u \bigl|_{C_R} \in H^{1}_{\alpha}(C_R)\ \mathrm{for\ all\ R>0} \}$. Then, we consider the following problem, which arises from taking the quasi-periodic Floquet Bloch transform (see, e.g., \cite{Lechleiter2}.) in (\ref{2.1})--(\ref{2.2}): For $\alpha \in [-1/2, 1/2]$, determine $u_{\alpha} \in H^{1}_{\alpha, loc}(C_{\infty})$ such that
\begin{equation}
\Delta u_{\alpha}+k^2nu_\alpha=f_{\alpha} \ \mathrm{in} \ C_{\infty}. \label{2.5}
\end{equation}
\begin{equation}
u_\alpha=0 \ \mathrm{on} \ (0,2\pi)\times \{0 \}. \label{2.6}
\end{equation}
Here, it is a natural to impose the {\it Rayleigh expansion} of the form 
\begin{equation}
u_{\alpha}(x)=\sum_{n \in \mathbb{Z}}u_{n}(\alpha)e^{inx_1+i\sqrt{k^2-(n+\alpha)^2}(x_2-h)}, \  x_2>h, \label{2.7}
\end{equation}
where $u_{n}(\alpha):=(2\pi)^{-1}\int_{0}^{2\pi}u_{\alpha}(x_1,h)e^{-inx_1}dx_1$ are the Fourier coefficients of $u_{\alpha}(\cdot,h)$, and $\sqrt{k^2-(n+\alpha)^2}=i\sqrt{(n+\alpha)^2-k^2}$ if $n+\alpha>k$. But even with this expansion the uniqueness of this problem fails for some $\alpha \in [-1/2, 1/2]$. We call $\alpha$ {\it exceptional values} if there exists non-trivial solutions $u_{\alpha} \in H^{1}_{\alpha, loc}(C_{\infty})$ of (\ref{2.5})--(\ref{2.7}). We set $A_k:=\{\alpha \in [-1/2, 1/2]: \exists l \in \mathbb{Z} \ s.t. \ |\alpha+l|=k \}$, and make the following assumption:
\begin{ass}
For every $\alpha \in A_k$ the solution of $u_{\alpha} \in H^{1}_{\alpha, loc}(C_{\infty})$ of (\ref{2.5})--(\ref{2.7}) has to be zero.
\end{ass}
The following properties of exceptional values was shown in \cite{Kirsch and Lechleiter2}.

\begin{lem}
Let Assumption 2.1 hold. Then, there exists only finitely many exceptional values $\alpha \in [-1/2, 1/2]$. Furthermore, if $\alpha$ is an exceptional value, then so is $-\alpha$. Therefore, the set of exceptional values can be described by $\{\alpha_j:j\in J \}$ where some $J \subset \mathbb{Z}$ is finite and $\alpha_{-j}=-\alpha_j$ for $j \in J$. For each exceptional value $\alpha_j$ we define 
\begin{equation}
X_j:=\left\{ \phi \in H^{1}_{\alpha_j, loc}(C_{\infty}):\begin{array}{cc}
      \Delta \phi+k^2n\phi=0 \ \mathrm{in} \ C_{\infty},\ \ \phi=0 \ \mathrm{for} \ x_2=0, \\
      \phi \ \mathrm{satisfies \ the \ Rayleigh\ expansion}\ (\ref{2.7})
    \end{array}
\right\} \nonumber
\end{equation}
Then, $X_j$ are finite dimensional. We set $m_j=\mathrm{dim}X_j$. Furthermore, $\phi \in X_j$ is evanescent, that is, there exists $c>0$ and $\delta>0$ such that $|\phi(x)|, \ |\nabla \phi(x)|\leq ce^{-\delta |x_2|}$ for all $x\in C_{\infty}$.
\end{lem}
Next, we consider the following eigenvalue problem in $X_j$: Determine $d \in \mathbb{R}$ and $\phi \in X_j$ such that
\begin{equation}
-i\int_{C_{\infty}}\frac{\partial \phi}{\partial x_1}\overline{\psi} dx= dk\int_{C_{\infty}}n\phi \overline{\psi}dx,\label{2.8}
\end{equation}
for all $\psi \in X_j$. We denote by the eigenvalues $d_{l,j}$ and eigenfunction $\phi_{l,j}$ of this problem, that is,
\begin{equation}
-i\int_{C_{\infty}}\frac{\partial \phi_{l,j}}{\partial x_1}\overline{\psi} dx= d_{l,j}k\int_{C_{\infty}}n\phi_{l,j} \overline{\psi}dx, \label{2.9}
\end{equation}
for every $l=1,...,m_j$ and $j \in J$. We normalize the eigenfunction $\{\phi_{l,j}: l=1,...,m_j \}$ such that
\begin{equation}
k\int_{C_{\infty}}n\phi_{l,j} \overline{\phi_{l',j}}dx=\delta_{l,l'},\label{2.10}
\end{equation}
for all $l, l'$. We will assume that the wave number $k>0$ is {\it regular} in the following sense.
\begin{dfn}
$k>0$ is {\it regular} if $d_{l,j}\neq 0$ for all $l=1,...m_j$ and $j \in J$.
\end{dfn}
Now we are ready to define the radiation condition. 
\begin{dfn}
Let Assumptions 2.1 hold, and let $k>0$ be regular in the sense of Definition 2.3. We set
\begin{equation}
\psi^{\pm}(x_1):=\frac{1}{2} \left[ 1\pm \frac{2}{\pi}\int_{0}^{x_1/2}\frac{sint}{t}dt \right] , \ x_1 \in \mathbb{R}.\label{2.11}
\end{equation}
Then, $u \in H^{1}_{loc}(\mathbb{R}^2_{+})$ satisfies the {\it radiation condition} if $u$ satisfies the upward propagating radiation condition (\ref{2.4}), and has a decomposition in the form $u=u^{(1)}+u^{(2)}$ where $u^{(1)} \bigl|_{\mathbb{R} \times (0,R)} \in H^{1}(\mathbb{R} \times (0,R))$ for all $R>0$, and $u^{(2)}\in L^{\infty}(\mathbb{R}^{2}_{+})$ has the following form
\begin{equation}
u^{(2)}(x)=\psi^{+}(x_1)\sum_{j \in J} \sum_{d_{l,j}>0}a_{l,j}\phi_{l,j}(x)+\psi^{-}(x_1)\sum_{j \in J} \sum_{d_{l,j}<0}a_{l,j}\phi_{l,j}(x) \label{2.12}
\end{equation}
where some $a_{l,j} \in \mathbb{C}$, and $\{d_{l,j},\phi_{l,j}: l=1,...,m_j \}$ are normalized eigenvalues and eigenfunctions of the problem (\ref{2.8}). 
\end{dfn}
\begin{rem}
It is obvious that we can replace $\psi^{+}$ by any smooth functions $\tilde{\psi}^{\pm}$ with $\tilde{\psi}^{+}(x_1)=1+\mathcal{O}(1/x_1)$ as $x_1\to \infty$ and $\tilde{\psi}^{+}(x_1)=\mathcal{O}(1/x_1)$ as $x_1\to -\infty$ and $\frac{d}{dx_1}\tilde{\psi}^{+}(x_1)\to 0$ as $|x_1|\to \infty$ (and analogously for $\psi^{-}$).
\end{rem}
The following was shown in Theorems 2.2, 6.6, and 6.8 of \cite{Kirsch and Lechleiter2}.
\begin{thm}
For every $f \in L^{2}(\mathbb{R}^2_{+})$ with the compact support in $W$, there exists a unique solution $u_{k+i \epsilon} \in H^{1}(\mathbb{R}^{2}_{+})$ of the problem (\ref{2.1})--(\ref{2.2}) replacing $k$ by $k+i\epsilon$. Furthermore, $u_{k+i \epsilon}$ converge as $\epsilon \to +0$ in $H^{1}_{loc}(\mathbb{R}^{2}_{+})$ to some $u \in H^{1}_{loc}(\mathbb{R}^{2}_{+})$ which satisfy (\ref{2.1})--(\ref{2.2}) and the radiation condition in the sense of Definition 2.4. Furthermore, the solution $u$ of this problem is uniquely determined.
\end{thm}
Furthermore, with the same radiation condition and the following additional assumption, the well-posedness of the perturbed scattering problem of (\ref{2.1})--(\ref{2.2}) was show in \cite{Furuya1}.
\begin{ass}
We assume that $k^2$ is not the point spectrum of $\frac{1}{(1+q)n}\Delta$ in $H^{1}_{0}(\mathbb{R}^{2}_{+})$, that is, evey $v \in H^{1}(\mathbb{R}^{2}_{+})$ which satisfies
\begin{equation}
\Delta v+k^2(1+q)nv=0 \ \mathrm{in} \ \mathbb{R}^2_{+}, \label{2.14}
\end{equation}
\begin{equation}
v=0 \ \mathrm{on} \ \Gamma_0, \label{2.15}
\end{equation}
has to vanishes for $x_2>0$.
\end{ass}
\begin{thm}
Let Assumption 2.7 hold and let $f \in L^{2}(\mathbb{R}^2_{+})$ such that $\mathrm{supp}f=Q$. Then, there exists a unique solution $u \in H^{1}_{loc}(\mathbb{R}^2_{+})$ such that
\begin{equation}
\Delta u+k^2(1+q)nu=f \ \mathrm{in} \ \mathbb{R}^2_{+}, \label{2.16}
\end{equation}
\begin{equation}
u=0 \ \mathrm{on} \ \Gamma_0, \label{2.17}
\end{equation}
and $u$ satisfies the radiation condition in the sense of Definition 2.4.
\end{thm}
By Theorem 2.8, the well-posedness of the perturbed scattering problem (\ref{1.1})--(\ref{1.2}) with the radiation condition follows. Then, we are able to consider the inverse problem of determining the support of $q$ from measured scattered field $u^s$ by the incident field $u^{i}(x,y)=\overline{G_n(x,y)}$. In the following sections we will discuss the inverse problem.
\section{The factorization method}
In Section 3, we mention the exact functional analytic theorem in ($F_{\#}$--) factorization method. The following functional analytic theorem is given by the almost same argument in Theorem 2.15 of \cite{Kirsch and Grinberg}.
\begin{thm}
Let $X \subset U\subset X^{*}$ be a Gelfand triple with a Hilbert space $U$ and a reflexive Banach space $X$ such that the imbedding is dense. Furthermore, let Y be a second Hilbert space and let $F:Y \to Y$, $G:X \to Y$, $T:X^{*} \to X$ be linear bounded operators such that 
\begin{equation}
F=GTG^{*}.\label{3.1}
\end{equation}
We make the following assumptions:
\begin{description}

\item[(1)] $G$ is compact with dense range in $Y$.

\item[(2)] There exists $t\in [0,2 \pi]$ such that $\mathrm{Re}(\mathrm{e}^{it}T)$ has the form $\mathrm{Re}(\mathrm{e}^{it}T)=C+K$ with some compact operator $K$ and some self-adjoint and positive coercive operator $C$, i.e., there exists $c>0$ such that
\begin{equation}
\langle \varphi,  C \varphi \rangle \geq c \left\| \varphi \right\|^2 \ for \ all \ \varphi \in X^{*}. \label{3.2}
\end{equation}

\item[(3)]
$\mathrm{Im} \langle \varphi, T \varphi \rangle > 0$ for all $\varphi \in \overline{\mathrm{Ran}(G^{*})}$ with $\varphi \neq 0$.
\end{description}
Then, the operator $F_{\#}:=\bigl|\mathrm{Re}(\mathrm{e}^{it}F)\bigr|+\mathrm{Im}F$ is non-negative, and the ranges of $G:X \to Y$ and $F_{\#}^{1/2}:Y \to Y$ coincide with each other, that is, we have the following range identity;
\begin{equation}
\mathrm{Ran}(G)=\mathrm{Ran}(F_{\#}^{1/2}).\label{3.3}
\end{equation}
\end{thm}
Here, the real part and the imaginary part of an operator $A$ are self-adjoint operators given by
\begin{equation}
\mathrm{Re}(A)=\displaystyle \frac{A+A^{*}}{2} \ \ \ \mathrm{and} \ \ \ \mathrm{Im}(A)=\displaystyle \frac{A-A^{*}}{2i}.\label{3.4}
\end{equation}
\begin{rem}
Here, we will mention a mistake of Theorem 2.1 in \cite{Lechleiter1}. It was introduced in order to avoid that $k^{2}$ is not a transmission eigenvalue corresponding to the unknown medium. To realize it, we replaced the assumptions (3) by the injectivity of $T$. However, its condition is not enough to obtain the range identity (\ref{3.3}). The following matrixes are its counterexample in which the strictly positivity of $\mathrm{Im}T$ is missing and the range identity (\ref{3.3}) fails.
\begin{equation}
\left(\begin{array}{cccc}
      1 &0 &0 &0 \\
      0 &1 &0 &0  
    \end{array}\right)
    \left(\begin{array}{cccc}
      0 &0 &0 &1 \\
      0 &0 &1 &0  \\
      0 &1 &0 &0 \\
      1 &0 &0 &0 
    \end{array}
\right)\left(\begin{array}{cc}
      1 &0 \\
      0 &1 \\
      0 &0 \\
      0 &0 
    \end{array}\right)
    =\left(\begin{array}{cc}
      0 &0 \\
      0 &0 
    \end{array}\right)
    ,\label{3.5}
\end{equation}
\begin{equation}
\mathrm{Ran}\left(\begin{array}{cccc}
      1 &0 &0 &0 \\
      0 &1 &0 &0  
    \end{array}\right)\neq \mathrm{Ran}\left(\begin{array}{cc}
      0 &0 \\
      0 &0 
    \end{array}\right).\label{3.6}
\end{equation}
From this counterexample one can not expect the range identity without $\mathrm{Im}T>0$. Therefore, in the factorization method for inverse medium scattering problem, we have to assume that $k^{2}$ is not a transmission eigenvalue corresponding to the unknown medium in order to have the strictly positivity of $\mathrm{Im}T$. 
\par 
In this section, we will prove Theorem 3.1 based on Theorem 2.15. We remark that in Theorem 2.15 of \cite{Kirsch and Grinberg} we assumed that $\mathrm{Im}T$ is compact, while in Theorem 3.1 of this section we will not assume its compactness. (The operator $\mathrm{Im}T$ is not compact in the case of inverse medium scattering problem with complex valued contrast function, See Theorem 4.5 of \cite{Kirsch and Grinberg}.) Before the proof of Theorem 3.1, we will show the following lemma.
\begin{lem}
Let X be a Hlbert space, and let $T: X \to X$ be a linear bounded, and let $K: X \to X$ be a linear bounded injective.
We assume that
\begin{equation}
\mathrm{Ran}(T) \ \mathrm{is \ closed \ subspace \ in} \ X, \ \  \mathrm{and} \ \ \mathrm{dim Ker}(T) <\infty.
\end{equation}
Then, there is a constant $C>0$ such that 
\begin{eqnarray}
\left\| u \right\|^2_{X} \leq
C (\left\| Tu \right\|^{2}_{X}+\left\| Ku \right\|^{2}_{X}) \ \ \ for \ all \ u \in X.
\end{eqnarray}

\end{lem}
\begin{proof}[{\bf Proof of Lemma 3.3}]
Assume that on contrary for any $C>0$, there exists a $u_c \in X$ such that
\begin{eqnarray}
\left\| u_c \right\|^2_{X} >
C (\left\| Tu_c \right\|^2_{X}+\left\| Ku_c \right\|^2_{X}).
\end{eqnarray}
Then, we can choose a sequence $\{u_m \} \subset X$ with $\left\| u_m \right\|^2=1$ such that $\left\| Tu_m \right\|^2+\left\| Ku_m \right\|^2$ converge to zero as $m \to \infty$. Since $\mathrm{Ker}(T)$ is finite dimensional subspace in $X$, there exists an orthogonal complement $\mathrm{Ker}(T)^{\bot}$ of $\mathrm{Ker}(T)$ such that $X=\mathrm{Ker}(T)\oplus \mathrm{Ker}(T)^{\bot}$. Since $\mathrm{Ker}(T)^{\bot}$ and $\mathrm{Ran}(T)$ is closed subspaces in $X$, the restrict operator $T\big|_{\mathrm{Ker}(T)^{\bot}}$ is injective and surjective from the Banach space $\mathrm{Ker}(T)^{\bot}$ to the Banach space $\mathrm{Ran}(T)$. Then by the closed graph theorem, $T\big|_{\mathrm{Ker}(T)^{\bot}}$ is invertible bounded, which implies that there is a constant $C>0$ such that
\begin{equation}
\left\| Pu_m \right\|^2\leq C\left\| T\big|_{\mathrm{Ker}(T)^{\bot}}Pu_m \right\|^2=C\left\| Tu_m \right\|^2.\label{bbb}
\end{equation}
Since $K$ is injective and $\mathrm{Ker}(T)$ is finite dimensional subspace in $X$, there is a constant $C>0$ such that
\begin{equation}
\left\| v \right\| \leq C\left\| Kv \right\|^2 \ \ \ for \ all \ v \in \mathrm{Ker}(T).
\end{equation}
(If not, we can take a sequence $\{v_m \}\subset \mathrm{Ker}(T)$ with $\left\| v_m \right\|=1$ and $\left\| Kv_m \right\|\to 0$. Since $\mathrm{Ker}(T)$ is finite dimensional subspace, there exists $v_0 \in \mathrm{Ker}(T)$ such that $v_m \to v_0$ and $Kv_m \to Kv_0$ as $m \to \infty$, which implies that $v_0=0$ by the injectivity of $K$. This contradicts with $\left\| v_m \right\|=1$.)
\par
Then, there is a constant $C>0$ such that
\begin{eqnarray}
\left\| (I-P)u_m \right\|^2 &\leq&
C \left\| K(I-P)u_m \right\|^{2} \leq 2(C \left\| Ku_m \right\|^{2}+ \left\| KPu_m \right\|^{2})\nonumber\\
&\leq&
2C (\left\| Ku_m \right\|^{2}+ \left\| K\right\|^{2}\left\|Pu_m \right\|^{2})\label{ccc}.
\end{eqnarray}
Therefore, by (\ref{bbb}) and (\ref{ccc}) there exists a constant $C'>0$ such that
\begin{equation}
1=\left\| Pu_m \right\|^2+\left\| (I-P)u_m \right\|^2 \leq C'(\left\| Tu_m \right\|^2+\left\| Ku_m \right\|^2).
\end{equation}
As $m \to \infty$ the right hand side of above inequality converges to zero, which contradicts.
\end{proof}
We will show Theorem 3.1.
\begin{proof}[{\bf Proof of Theorem 3.1}]
By the same argument of Part A (Reduction) in the proof of Theorem 2.15 (\cite{Kirsch and Grinberg}), we can restrict ourselves to the case $X=U$ and $C=I$. Furthermore, we can also  restrict ourselves to the case $G$ is injective. Indeed, let $P: U \to \overline{\mathrm{Ran}(G^{*})}$ be the orthogonal projection onto $\hat{U}:=\overline{\mathrm{Ran}(G^{*})} \subset U$. Then, $PG^{*}=G^{*}$ and $G=GP$. By this, we can have the factorization of the form
\begin{equation}
F=GPTPG^*=\hat{G}\hat{T}\hat{G^*}
\end{equation}
where $\hat{G}\bigl|_{\hat{U}}:\hat{U} \to Y$ and $\hat{T}=PT\bigl|_{\hat{U}}: \hat{U} \to \hat{U}$. Therefore, all of assumptions (1)--(3) are satisfied. (We remark that $\hat{T}$ is not injective even if $T$ is injective, which leads to error in Theorem 2.1 of \cite{Lechleiter1}.)
\par
By the same argument in Part B, C, and D in the proof of Theorem 2.15 (\cite{Kirsch and Grinberg}), we can show that 
\begin{equation}
F_{\#}=GT_{\#}\hat{G^*}
\end{equation}
where $T_{\#}=\mathrm{Re}(\mathrm{e}^{it}T)D+\mathrm{Im}T$ and $D$ is an isomorphism from $U$ onto itself. It was shown that the operator $T_{\#}$ is non-negative on $U$ in the proof. By applying the inequality (4.5) of \cite{Kirsch} to the non-negative operators $\mathrm{Re}(\mathrm{e}^{it}T)D$ and $\mathrm{Im}T$, there is a constant $C>0$ such that 
\begin{eqnarray}
\langle \varphi,  T_{\#} \varphi \rangle
&=&
\langle \varphi,  \mathrm{Re}(\mathrm{e}^{it}T)D \varphi \rangle
+\langle \varphi,  \mathrm{Im}T \varphi \rangle \nonumber\\
&\geq&
C (\left\| \mathrm{Re}(\mathrm{e}^{it}T)D  \varphi \right\|^2+\left\| \mathrm{Im}T \varphi \right\|^2) \ \ \  \ for \ all \ \varphi \in U.
\end{eqnarray}
By assumption (2) $\mathrm{Re}(\mathrm{e}^{it}T)D$ is a Fredholm operator, and by assumption (3) $\mathrm{Im}T$ is injective. Therefore by applying Lemma 3.3 to our operators, there is a constant $C>0$ such that
\begin{eqnarray}
C(\left\| \mathrm{Re}(\mathrm{e}^{it}T)D  \varphi \right\|+\left\| \mathrm{Im}T \varphi \right\|)
\geq
\left\| \varphi \right\| \ \ \  \ for \ all \ \varphi \in U,
\end{eqnarray}
which implies that the operator $T_{\#}:U\to U$ is positive coercive. Since we can write
\begin{equation}
F_{\#}=F_{\#}^{1/2}(F_{\#}^{1/2})^{*}=GT_{\#}G^{*},
\end{equation}
then by applying Theorem 1.21 of \cite{Kirsch and Grinberg}, we have the range of $F_{\#}^{1/2}$ and $G$ coincide. We have shown Theorem 3.1. 
\end{proof}
\end{rem}
\section{A factorization of the near field operator}
In Section 4, we discuss a factorization of the near field operator $N$. We define the operator $L:L^{2}(Q)\to L^{2}(M)$ by $Lf:=v\bigl|_{M}$ where $v$ satisfies the radiation condition in the sense of Definition 2.4 and 
\begin{equation}
\Delta v+k^2(1+q)nv=-k^{2}\frac{nq}{\sqrt{|nq|}}f, \ \mathrm{in} \ \mathbb{R}^2_{+}, \label{4.1}
\end{equation}
\begin{equation}
v=0 \ \mathrm{on} \ \mathbb{R}\times \{0\}. \label{4.2}
\end{equation}
We define $H:L^{2}(M)\to L^{2}(Q)$ by
\begin{equation}
Hg(x):=\sqrt{|n(x)q(x)|}\int_{M}\overline{G_n(x,y)}g(y)ds(y), \ x \in Q.\label{4.3}
\end{equation}
Then, by these definition we have $N=LH$. In order to make a symmetricity of the factorization of the near field operator $N$, we will show the following symmetricity of the Green function $G_n$.
\begin{lem}
\begin{equation}
G_n(x,y)=G_n(y,x), \ x\neq y.\label{4.4}
\end{equation}
\end{lem}
\begin{proof}[{\bf Proof of Lemma 4.1}]
We take a small $\eta>0$ such that $B_{2\eta}(x) \cap B_{2\eta}(y) =\emptyset$ where $B_\epsilon(z)\subset \mathbb{R}^{2}$ is some open ball with center $z$ and radius $\epsilon >0$. We recall that $G_n(z,y)=G(z,y)+\tilde{u}^{s}(z,y)$ where $G(z,y)=\Phi_k(z,y)-\Phi_k(z,y^{*})$ and $\tilde{u}^{s}(z,y)$ is a radiating solution of the problem (\ref{1.5}) such that $\tilde{u}^{s}(z,y)=0$ for $z_2=0$. In Introduction of \cite{Kirsch and Lechleiter2} $\tilde{u}^{s}$ is given by $\tilde{u}^{s}(z,y)=u(z,y)-\chi(|z-y|)G(z,y)$ where $\chi \in C^{\infty}(\mathbb{R}_+)$ satisfying $\chi(t)=0$ for $0\leq t\leq \eta/2$ and $\chi(t)=1$ for $t\geq \eta$, and $u$ is a radiating solution such that $u=0$ on $\mathbb{R}\times \{0\}$ and 
\begin{equation}
\Delta u+k^2nu=f(\cdot,y)\ \mathrm{in} \ \mathbb{R}^2_{+}, \label{4.5}
\end{equation}
\begin{equation}
u=0 \ \mathrm{on} \ \mathbb{R}\times\{0\}, \label{4.6}
\end{equation}
where
\begin{equation}
f(\cdot,y):=\Bigl[k^2\bigl(1-n\bigr)\bigl(1-\chi(|\cdot-y|)\bigr)+\Delta \chi(|\cdot-y|) \Bigr]G(\cdot,y)+2\nabla \chi(|\cdot-y|)\cdot \nabla G(\cdot,y). \label{4.7}
\end{equation}
Then, we have $G_n(z,y)=u(z,y)+(1-\chi(|z-y|))G(z,y)$. By Theorem 2.6 we can take an solution $u_{\epsilon} \in H^{1}(\mathbb{R}^{2}_{+})$ of the problem (\ref{4.5})--(\ref{4.6}) replacing $k$ by $(k+i\epsilon)$ satisfying $u_{\epsilon}$ converges as $\epsilon \to +0$ in $H^{1}_{loc}(\mathbb{R}^2_{+})$ to $u$. We set $G_{n,\epsilon}(z,y):=u_{\epsilon}(z,y)+(1-\chi(|z-y|))G(z,y)$, and $G_{n,\epsilon}(z,y)$ converges as $\epsilon \to +0$ to $G(z,y)$  pointwise for $z \in \mathbb{R}^{2}_{+}$. By the simple calculation, we have
\begin{equation}
\bigl[\Delta_z + (k+i\epsilon)^2 n(z)\bigr]G_{n,\epsilon}(z,y)=-\delta(z,y)+(2k\epsilon i-\epsilon^2)n(z)\bigl(1-\chi(|z-y|)\bigr)G(z,y).\label{4.8}
\end{equation}
\par
Let $r>0$ be large enough such that $x,y \in B_r(0)$. By Green's second theorem in $B_r(0)\cap \mathbb{R}^{2}_+$ we have
\begin{eqnarray}
\lefteqn{-G_{n,\epsilon}(y,x)+(2k\epsilon i-\epsilon^2)\int_{ B_{2\eta}(y)}u_{\epsilon}(z,x)n(z)(1-\chi(|z-y|))G(z,y)dz}
\nonumber\\
&+&G_{n,\epsilon}(x,y)-(2k\epsilon i-\epsilon^2)\int_{ B_{2\eta}(x)}u_{\epsilon}(z,y)n(z)(1-\chi(|z-x|))G(z,x)dz
\nonumber\\
&=&\int_{B_r(0)\cap \mathbb{R}^{2}_+} G_{n,\epsilon}(z,x)\bigl[\Delta_z+(k+i\epsilon)^2n(z)\bigr]G_{n,\epsilon}(z,y)dz
\nonumber\\
&-&\int_{B_r(0)\cap \mathbb{R}^{2}_+}G_{n,\epsilon}(z,y)\bigl[\Delta_z+(k+i\epsilon)^2n(z)\bigr]G_{n,\epsilon}(z,x)dz
\nonumber\\
&=&\int_{\partial B_r(0)\cap \mathbb{R}^{2}_+} u_{\epsilon}(z,x)\frac{\partial u_{\epsilon}(z,y)}{\partial \nu_z}-u_{\epsilon}(z,y)\frac{\partial u_{\epsilon}(z,x)}{\partial \nu_z}ds(z).
\label{4.9} 
\end{eqnarray}
Since $u_{\epsilon} \in H^{1}(\mathbb{R}^2_+)$, the right hand side of (\ref{4.9}) converges as $r \to \infty$ to zero. Then, as $r \to \infty$ in (\ref{4.9}) we have
\begin{eqnarray}
\lefteqn{G_{n,\epsilon}(x,y)-G_{n,\epsilon}(y,x)}
\nonumber\\
&=&(2k\epsilon i-\epsilon^2)\int_{B_{2\eta}(x)}u_{\epsilon}(z,y)n(z)(1-\chi(|z-x|))G(z,x)dz
\nonumber\\
&-&(2k\epsilon i-\epsilon^2)\int_{B_{2\eta}(y)}u_{\epsilon}(z,x)n(z)(1-\chi(|z-y|))G(z,y)dz
\label{4.10} 
\end{eqnarray}
Since $u_{\epsilon}$ converges as $\epsilon \to +0$ in $H^{1}_{loc}(\mathbb{R}^2_{+})$ to $u$, the right hand side of (\ref{4.10}) converges to zero as $\epsilon \to +0$. Therefore, we conclude that $G_{n}(x,y)=G_{n}(y,x)$ for $x\neq y$.
\end{proof}
By the symmetricity of $G_n$,
\begin{eqnarray}
\langle Hg, f \rangle
&=&\int_{Q} \bigl\{ \sqrt{|n(x)q(x)|}\int_{M}\overline{G_n(x,y)}g(y)ds(y) \bigr\}\overline{f(x)}dx
\nonumber\\
&=& \int_{M} g(y) \overline{ \bigl\{ \int_{Q}\sqrt{|n(x)q(x)|}G_n(x,y)f(x)ds(x) \bigr\} }ds(y)
\nonumber\\
&=& \int_{M} g(y) \overline{\bigl\{ \int_{Q}\sqrt{|n(x)q(x)|}G_n(y,x)f(x)ds(x) \bigr\}}ds(y), \ \ \ \ 
\label{4.11}
\end{eqnarray}
which implies that 
\begin{equation}
H^{*}f(x)=\int_{Q}\sqrt{|n(y)q(y)|}G_n(x,y)f(y)ds(y), \ x \in M.\label{4.12}
\end{equation}
We define $T:L^{2}(Q) \to L^{2}(Q)$ by $Tf:=\frac{|nq|}{k^2nq}f-\sqrt{|nq|}w$ where $w$ satisfies the radiation condition and 
\begin{equation}
\Delta w+k^2nw=-\sqrt{|nq|}f, \ \mathrm{in} \ \mathbb{R}^2_{+}, \label{4.13}
\end{equation}
\begin{equation}
v=0 \ \mathrm{on} \ \mathbb{R}\times \{0\}. \label{4.14}
\end{equation} 
We will show the following integral representation of $w$.
\begin{lem}
\begin{equation}
w(x)=\int_{Q}\sqrt{|n(y)q(y)|}G_n(x,y)f(y)dy, \ x \in \mathbb{R}^2_{+}.\label{4.15}
\end{equation}
\end{lem}
\begin{proof}[{\bf Proof of Lemma 4.2}]
Let $w_{\epsilon} \in H^{1}_{loc}(\mathbb{R}^2_+)$ be a solution of the problem (\ref{4.13})--(\ref{4.14}) replacing $k$ by $(k+i \epsilon)$ satisfying $w_{\epsilon}$ converges as $\epsilon \to +0$ in $H^{1}_{loc}(\mathbb{R}^{2}_+)$ to $w$. Let $G_{n,\epsilon}(y,x)$ be an approximation of the Green's function $G_n(y,x)$ as same as in Lemma 4.1. Let $r>0$ be large enough such that $x \in B_r(0)$. By Green's second theorem in $B_r(0) \cap \mathbb{R}^{2}_+$ we have
\begin{eqnarray}
\lefteqn{-w_{\epsilon}(x)+(2k\epsilon i-\epsilon^2)\int_{B_{2\eta}(x)}w_{\epsilon}(y)n(y)(1-\chi(|y-x|))G(y,x)dy}
\nonumber\\
&+&\int_{Q}\sqrt{|n(y)q(y)|}G_{n,\epsilon}(y,x)f(y)dy
\nonumber\\
&=&\int_{B_r(0)\cap \mathbb{R}^{2}_+}w_{\epsilon}(y)\bigl[\Delta_y+(k+i\epsilon)^2n(y)\bigr]G_{n,\epsilon}(y,x)dy
\nonumber\\
&-&\int_{B_r(0)\cap \mathbb{R}^{2}_+}G_{n,\epsilon}(y,x)\bigl[\Delta_y+(k+i\epsilon)^2n(y)\bigr]w_{\epsilon}(y)dz
\nonumber\\
&=&\int_{\partial B_r(0)\cap \mathbb{R}^{2}_+} w_{\epsilon}(y)\frac{\partial u_{\epsilon}(y,x)}{\partial \nu_y}-u_{\epsilon}(y,x)\frac{\partial w_{\epsilon}(y)}{\partial \nu_y}ds(y).
\label{4.16} 
\end{eqnarray}
Since $u_{\epsilon}$, $w_{\epsilon}$ $\in H^{1}(\mathbb{R}^2_+)$, the right hand side of (\ref{4.16}) converges as $r \to \infty$ to zero. Then, as $r \to \infty$ in (\ref{4.16}) we have
\begin{eqnarray}
w_{\epsilon}(x)&=&(2k\epsilon i-\epsilon^2)\int_{B_{2\eta}(x)}w_{\epsilon}(y)n(y)(1-\chi(|y-x|))G(y,x)dy
\nonumber\\
&+&\int_{Q}\sqrt{|n(y)q(y)|}G_{n,\epsilon}(y,x)f(y)dy
\label{4.17} 
\end{eqnarray}
The first term of right hand side in (\ref{4.17}) converges to zero as $\epsilon \to +0$, and the second term converges to $\int_{Q}\sqrt{|n(y)q(y)|}G_{n}(y,x)f(y)dy$ as $\epsilon \to +0$. As $\epsilon \to +0$ in (\ref{4.17}) and by the symmetricity of $G_n$ (Lemma 4.1) we conclude (\ref{4.15}). 
\end{proof}
Since $w$ satisfies
\begin{eqnarray}
\Delta w+k^2(1+q)nw
&=&-k^2\frac{nq}{\sqrt{|nq|}}\bigl\{ \frac{|nq|}{k^2nq}f-\sqrt{|nq|}w \bigr\} \ \mathrm{in} \ \mathbb{R}^2_{+}
\nonumber\\
&=&-k^2\frac{nq}{\sqrt{|nq|}}Tf
, \label{4.18}
\end{eqnarray}
we have $w\bigl|_{M}=LTf$. Therefore, by (\ref{4.12}) and (\ref{4.15}) we have $H^{*}=LT$. Then, we have the following symmetric factorization:
\begin{equation}
N=LT^{*}L^{*}.\label{4.19}
\end{equation}
We will show the following lemma corresponding to the assumptions in Theorem 3.1.
\begin{lem}
\begin{description}
\item[(a)] $L$ is compact with dense range in $L^{2}(M)$.

\item[(b)]If there exists the constant $q_{min}>0$ such that $q_{min}\leq q$ a.e. in $Q$, then  $\mathrm{Re}T$ has the form $\mathrm{Re}T=C+K$ with some self-adjoint and positive coercive operator $C$ and some compact operator $K$ on $L^{2}(Q)$.

\item[(c)]$\mathrm{Im} \langle f, T f \rangle \geq 0$ for all $f \in L^{2}(Q)$.

\item[(d)]
$T$ is injective.
\end{description}
\end{lem} 
\begin{proof}[{\bf Proof of Lemma 4.3}]
{\bf(d)} Let $f \in L^{2}(Q)$ and $Tf=0$, i.e., $\frac{|nq|}{k^2nq}f=\sqrt{|nq|}w$ where $w$ satisfies (\ref{4.13})--(\ref{4.14}). Then, $\Delta w+k^2n(1+q)w=0$. By the uniqueness, $w=0$ in $\mathbb{R}^2_{+}$ which implies that $f=0$. Therefore $T$ is injective. 
\par
{\bf(b)} Since $n$ and $q$ are bounded below (that is, $n\geq n_{min}>0$ and $q\geq q_{min}>0$), $T$ has the form $T=C+K$ where $K$ is some compact operator and $C$ is some self-adjoint and positive coercive operator. Furthermore, from the injectivity of $T$ we obtain that $T$ is bijective. 
\par
{\bf(a)} By the trace theorem and $v \in H^{1}_{loc}(\mathbb{R}^2_{+})$, $Lf=v\bigl|_{M} \in H^{1/2}(M)$, which implies that $L:L^{2}(Q)\to L^{2}(M)$ is compact. 
\par
By the bijectivity of $T$ and $H=T^{*}L^{*}$, it is sufficient to show the injectivity of $H$. Let $g \in L^{2}(M)$ and $Hg(x)=\sqrt{|n(x)q(x)|}\int_{M}\overline{G_n(x,y)}g(y)ds(y)=0$ for $\ x \in Q$. We set $v(x):=\int_{M}\overline{G_n(x,y)}g(y)ds(y)$. By the definition of $v$ we have
\begin{equation}
\Delta v+k^2nv=0, \ \mathrm{in} \ \mathbb{R}^2_{+}\setminus M , \label{4.20}
\end{equation}
and since $q$ are bounded below, $v=0$ in $Q$. By unique continuation principle we have $v=0$ in $\mathbb{R}^2_{+}\setminus M$. By the jump relation (see \cite{McLean}) we have $0=\frac{\partial v_+}{\partial \nu}-\frac{\partial v_-}{\partial \nu}=g$, which conclude that the operator $H$ is injective.
\par
{\bf(c)}
For the proof of (c) we refer to Theorem 3.1 in \cite{Furuya1}. By the definition of $T$ we have 
\begin{equation}
\mathrm{Im}\langle f, Tf \rangle=-\mathrm{Im}\int_{Q}f\sqrt{|nq|}\overline{w}dx=\mathrm{Im}\int_{Q}\overline{w}[\Delta+k^2n]wdx,\label{4.21}
\end{equation}
where $w$ is a radiating solution of the problem (\ref{4.13})--(\ref{4.14}). We set $\Omega_N:=(-N,N) \times (0, N^s)$ where $s>0$ is small enough and $N>0$ is large enough. By the same argument in Theorem 3.1 of \cite{Furuya1} we have
\begin{eqnarray}
\lefteqn{\mathrm{Im}\langle f, Tf \rangle=\mathrm{Im}\int_{\Omega_N}\overline{w}[\Delta+k^2n]wdx=\mathrm{Im}\int_{\Omega_N}\overline{w}\Delta wdx}
\nonumber\\
&\geq&
\Biggl[\frac{1}{2\pi} \sum_{j \in J} \sum_{d_{l,j},d_{l',j}>0}\overline{a_{l,j}}a_{l',j}\int_{C_{\phi(N)}}\overline{\phi_{l,j}}\frac{\partial \phi_{l',j}}{\partial x_1}dx \Biggr]
\nonumber\\
&-&\mathrm{Im}\Biggl[\frac{1}{2\pi} \sum_{j \in J}\sum_{d_{l,j},d_{l',j}<0}\overline{a_{l,j}}a_{l',j}\int_{C_{\phi(N)}}\overline{\phi_{l,j}}\frac{\partial \phi_{l',j}}{\partial x_1}dx\Biggr]+o(1),\label{4.22}
\end{eqnarray}
where where some $a_{l,j} \in \mathbb{C}$, and $\{d_{l,j},\phi_{l,j}: l=1,...,m_j \}$ are normalized eigenvalues and eigenfunctions of the problem (\ref{2.8}). By Lemmas 6.3 and 6.4 of \cite{Kirsch and Lechleiter2}, as $N\to \infty$ in (\ref{4.22}) we have 
\begin{equation}
\mathrm{Im}\langle f, Tf \rangle \geq \frac{k}{2\pi} \sum_{j \in J} \Biggl[ \sum_{d_{l,j}>0}|a_{l,j}|^2d_{l,j} -\sum_{d_{l,j}<0}|a_{l,j}|^2d_{l,j}\Biggr]\geq 0,\label{4.23}
\end{equation}
which concludes (c).
\end{proof}
\begin{rem}
The strictly positivity of $\mathrm{Im}T$ is missing in Lemma 4.3 although we have the injectivity of $T$. From the viewpoint of Section 3, the assumption of transmission eigenvalue of $Q$ can be expected when we apply Theorem 3.1 to this case. However, even with its assumption the author of this paper do not have the idea to show $\mathrm{Im}T>0$.
\end{rem}
In order to show Theorems 1.1 and 1.2, we consider another factorization of the near field operator $N$. We define $\tilde{T}:L^{2}(Q)\to L^{2}(Q)$ by $\tilde{T}v:=k^2\frac{nq}{|nq|}g-k^2\frac{nq}{\sqrt{|nq|}}v$ where $v$ satisfies the radiation condition and 
\begin{equation}
\Delta v+k^2(1+q)nv=-k^2\frac{nq}{\sqrt{|nq|}}g, \ \mathrm{in} \ \mathbb{R}^2_{+}, \label{4.24}
\end{equation}
\begin{equation}
v=0 \ \mathrm{on} \ \mathbb{R}\times \{0\}. \label{4.25}
\end{equation} 
Then, by the definition of $T$ and $\tilde{T}$ we can show that $\tilde{T}T=I$ and $T\tilde{T}=I$, which implies that $T^{-1}=\tilde{T}$. Therefore, we have by $L=H^{*}T^{-1}$
\begin{equation}
N=LT^{*}L^{*}=H^{*}T^{-1}H=H^{*}\tilde{T}H=H^{*}_{Q}\hat{T}H_{Q}, \label{4.26}
\end{equation}
where $H_{Q}:L^{2}(M)\to L^{2}(Q)$ is defined by 
\begin{equation}
H_{Q}g(x):=\int_{M}\overline{G_n(x,y)}g(y)ds(y), \ x \in Q, \label{4.27}
\end{equation} 
and $\hat{T}:L^{2}(Q)\to L^{2}(Q)$ is defined by $\hat{T}f=k^2nqf+k^2nqw$ where $w$ satisfies the radiation condition and 
\begin{equation}
\Delta w+k^2(1+q)nw=-k^2nqf, \ \mathrm{in} \ \mathbb{R}^2_{+}, \label{4.28}
\end{equation}
\begin{equation}
w=0 \ \mathrm{on} \ \mathbb{R}\times \{0\}. \label{4.29}
\end{equation}
We will show the following lemma.
\begin{lem}
Let $B$ and $Q$ be a bounded open set in $\mathbb{R}^{2}_+$.
\begin{description}
\item[(a)] $\mathrm{dim}(\mathrm{Ran}(H^{*}_B))=\infty$.

\item[(b)]If $B \cap Q =\emptyset$, then $\mathrm{Ran}(H^{*}_B) \cap \mathrm{Ran}(H^{*}_Q)=\{0 \}$.
\end{description}
\end{lem} 
\begin{proof}[{\bf Proof of Lemma 4.4}]
{\bf (a)} By the same argument of the injectivity of $H$ in (a) of Lemma 4.3, we can show that $H_B$ is injective. Therefore, $H^{*}_B$ has dense range.
\par
{\bf (b)} Let $h \in \mathrm{Ran}(H^{*}_B) \cap \mathrm{Ran}(H^{*}_Q)$. Then, there exists $f_B$, $f_Q$ suct that $h=H^{*}_Bf_B=H^{*}_Qf_Q$. We set
\begin{equation}
v_B(x):=\int_BG_n(x,y)f_B(y)dy, \ x \in \mathbb{R}^2_+\label{4.30}
\end{equation} 
\begin{equation}
v_Q(x):=\int_QG_n(x,y)f_Q(y)dy, \ x \in \mathbb{R}^2_+\label{4.31}
\end{equation} 
then, $v_B$ and $v_Q$ satisfies $\Delta v_B+k^2nv_B=-f_B$, and $\Delta v_Q+k^2nv_Q=-f_Q$, respectivelty, and $v_B=v_Q$ on $M$. By Rellich lemma and unique continuation we have $v_B=v_Q$ in $\mathbb{R}^2_{+}\setminus(\overline{B\cap Q})$. Hence, we can define $v \in H^{1}_{loc}(\mathbb{R}^2)$ by
\begin{eqnarray}
v:=\left\{ \begin{array}{ll}
v_B=v_Q & \quad \mbox{in $\mathbb{R}^2_{+}\setminus (\overline{B \cap Q})$}  \\
v_{B} & \quad \mbox{in $Q$} \\
v_{Q} & \quad \mbox{in $B$} \\
\end{array} \right.\label{4.32}
\end{eqnarray}
and $v$ is a radiating solution such that $v=0$ for $x_2=0$ and
\begin{equation}
\Delta v+k^2nv=0 \ \mathrm{in} \ \mathbb{R}^2_+. \label{4.33}
\end{equation}
By the uniquness, we have $v=0 \ \mathrm{in} \ \mathbb{R}^2$, which implies that $h=0$.   
\end{proof}
In the following sections we will show Theorems 1.1 and 1.2 by using properties of the factorization of the near field operator $N$.
\section{Proof of Theorem 1.1}
In Section 5, we will show Theorem 1.1. Let $B \subset Q$. We define $K:L^{2}(Q)\to L^{2}(Q)$ by $Kf:=k^2nqw$ where $w$ is a radiating solution of the problem (\ref{4.28})--(\ref{4.29}). Since $w\bigl|_Q \in H^{1}(Q)$, $K$ is a compact operator. Let $V$ be the sum of eigenspaces of $\mathrm{Re}K$ associated to eigenvalues less than $\alpha-k^2n_{min}q_{min}$. Since $\alpha-k^2n_{min}q_{min}<0$, then $V$ is a finite dimensional and for $H_Qg\in V^{\bot}$
\begin{eqnarray}
\langle \mathrm{Re}Ng, g \rangle&=&\int_{Q}k^2nq|H_Qg|^2
dx+\langle (\mathrm{Re}K)H_Qg, H_Qg \rangle\nonumber\\
&\geq&
k^2n_{min}q_{min}\left\| H_Qg \right\|^{2}+(\alpha-k^2n_{min}q_{min})\left\| H_Qg \right\|^{2}
\nonumber\\
&\geq&\alpha\left\| H_Qg \right\|^{2} \geq \alpha\left\| H_Bg \right\|^{2}
\label{5.1}
\end{eqnarray}
Since for $g \in L^{2}(M)$
\begin{equation}
H_Qg \in V^{\bot} \ \ \ \  \Longleftrightarrow \ \ \ \ g \in (H^{*}_QV)^{\bot},\label{5.2}
\end{equation}
and $\mathrm{dim}(H^{*}_QV) \leq \mathrm{dim}(V)< \infty$, we have by Corollary 3.3 of \cite{B. Harrach and V. Pohjola and M. Salo2} that $\alpha H^{*}_{B}H_{B} \leq_{\mathrm{fin}} \mathrm{Re}N$.
\par
Let now $B \not\subset Q$ and assume on the contrary $\alpha H^{*}_{B}H_{B}\leq_{\mathrm{fin}} \mathrm{Re}N$, that is, by Colrollary 3.3 of \cite{B. Harrach and V. Pohjola and M. Salo2} there exists a finite dimensional subspace $W$ in $L^{2}(M)$ such that
\begin{equation}
\langle (\mathrm{Re}N-\alpha H^{*}_{B}H_{B})w, w \rangle \geq 0,\label{5.3}
\end{equation}
for all $w \in W^{\bot}$. Since $B \not\subset Q$,  we can take a small open domain $B_0 \subset B$ such that $B_0 \cap Q= \emptyset$, which implies that for all $w \in W^{\bot}$ 
\begin{eqnarray}
\alpha \left\| H_{B_0}w \right\|^{2}
&\leq&
\alpha \left\| H_{B}w \right\|^{2}
\nonumber\\
&\leq&\langle (\mathrm{Re}N)w, w \rangle
\nonumber\\
&=&\langle (\mathrm{Re}\hat{T})H_Qw, H_Qw \rangle
\nonumber\\
&\leq&\left\| \mathrm{Re}\hat{T} \right\| \left\|H_Qw \right\|^{2}.
\label{5.4}
\end{eqnarray}
By (a) of Lemma 4.7 in \cite{B. Harrach and V. Pohjola and M. Salo2}, we have
\begin{equation}
\mathrm{Ran}(H^{*}_{B_0}) \not\subseteq \mathrm{Ran}(H_{Q}^{*})+W= \mathrm{Ran}(H_{Q}^{*},\ P_{W}),\label{5.5}
\end{equation}
where $P_{W}:L^{2}(M)\to L^{2}(M)$ is the orthogonal projection on $W$. Lemma 4.6 of \cite{B. Harrach and V. Pohjola and M. Salo2} implies that for any $C>0$ there exists a $w_c$ such that 
\begin{equation}
\left\|H_{B_0}w_{c} \right\|^2 >  C^{2}\left\| \left(
    \begin{array}{ccc}
      H_{Q} \\
      P_{V} \\
    \end{array}
\right) w_c \right\|^2=C^{2}(\left\|H_Qw_{c} \right\|^2+\left\|P_{W}w_{c} \right\|^2).\label{5.6}
\end{equation}
Hence, there exists a sequence $(w_m)_{m\in \mathbb{N}} \subset  L^{2}(\mathbb{S}^{1})$ such that $\left\|H_{B_0}w_{m} \right\| \to \infty$ and $\left\|H_Qw_{m} \right\|+\left\|P_{V}w_{m} \right\| \to 0$ as $m \to \infty$. Setting $\tilde{w}_{m}:=w_{m}-P_{W}w_{m} \in W^{\bot}$ we have as $m\to \infty$,
\begin{equation}
\left\|H_{B_0}\tilde{w}_{m} \right\| \geq \left\|H_{B_0}w_{m} \right\|-\left\|H_{B_0}\right\| \left\|P_{W}w_{m}\right\| \to \infty,\label{5.7}
\end{equation}
\begin{equation}
\left\|H_Q\tilde{w}_{m} \right\| \leq \left\|H_Qw_{m} \right\|+\left\| H_Q \right\| \left\|P_{W}w_{m} \right\| \to 0.\label{5.8}
\end{equation}
This contradicts (\ref{5.4}). Therefore, we have $\alpha H^{*}_{B}H_{B}\not\leq_{\mathrm{fin}} \mathrm{Re}N$. Theorem 1.1 has been shown. \qed
\vspace{4mm}
\par
By the same argument in Theorem 1.1 we can show the following.
\begin{cor}
Let $B \subset \mathbb{R}^2$ be a bounded open set. Let Assumption hold, and assume that there exists $q_{max}<0$ such that $q\leq q_{max}$ a.e. in $Q$. Then for $0<\alpha<k^2n_{min}|q_{max}|$, 
\begin{equation}
B \subset Q \ \ \ \  \Longleftrightarrow \ \ \ \  \alpha H^{*}_{B}H_{B}\leq_{\mathrm{fin}} -\mathrm{Re}N,\label{5.9}
\end{equation}
\end{cor}
\section{Proof of Theorem 1.2}
In Section 6, we will show Theorem 1.2. Let $Q \subset B$. Let $V$ be the sum of eigenspaces of $\mathrm{Re}K$ associated to eigenvalues more than $\alpha-k^2n_{max}q_{max}$. Since $\alpha-k^2n_{max}q_{max}>0$, then $V$ is a finite dimensional and for $H_Qg\in V^{\bot}$
\begin{eqnarray}
\langle \mathrm{Re}Ng, g \rangle&=&\int_{Q}k^2nq|H_Qg|^2
dx+\langle (\mathrm{Re}K)H_Qg, H_Qg \rangle\nonumber\\
&\leq&
k^n_{max}q_{max}\left\| H_Qg \right\|^{2}+(\alpha-k^2n_{max}q_{max})\left\| H_Qg \right\|^{2}
\nonumber\\
&\leq&\alpha\left\| H_Qg \right\|^{2}\leq \alpha\left\| H_Bg \right\|^{2}.
\label{6.1}
\end{eqnarray}
Since for $g \in L^{2}(M)$
\begin{equation}
H_Qg \in V^{\bot} \ \ \ \  \Longleftrightarrow \ \ \ \ g \in (H^{*}_QV)^{\bot},\label{6.2}
\end{equation}
and $\mathrm{dim}(H^{*}_QV) \leq \mathrm{dim}(V)< \infty$, we have by Corollary 3.3 of \cite{B. Harrach and V. Pohjola and M. Salo2} that $\mathrm{Re}N \leq_{\mathrm{fin}} \alpha H^{*}_{B}H_{B}$.
\par
Let now $Q \not\subset B$ and assume on the contrary $\mathrm{Re}N \leq_{\mathrm{fin}}\alpha H^{*}_{B}H_{B}$, that is, by Corollary 3.3 of \cite{B. Harrach and V. Pohjola and M. Salo2} there exists a finite dimensional subspace $W$ in $L^{2}(M)$ such that
\begin{equation}
\langle (\alpha H^{*}_{B}H_{B}-\mathrm{Re}N)w, w \rangle \geq 0,\label{6.3}
\end{equation}
for all $w \in W^{\bot}$. Since $Q \not\subset B$,  we can take a small open domain $Q_0 \subset Q$ such that $Q_0 \cap B= \emptyset$. Let $V$ be the sum of eigenspaces of $\mathrm{Re}K$ associated to eigenvalues less than $-k^2n_{min}q_{min}/2$. Then, $V$ is a finite dimensional and for $g \in (H^{*}_QV)^{\bot} \cap W^{\bot}=(H^{*}_QV \cup W)^{\bot}$ we have
\begin{eqnarray}
&&\alpha \left\| H_{B}g \right\|^{2}
\nonumber\\
&\geq&
\langle (\mathrm{Re}N)g, g \rangle
\nonumber\\
&=&\int_{Q}k^2nq|H_Qg|^2dx + \langle (\mathrm{Re}\hat{K})H_Qg, H_Qg \rangle
\nonumber\\
&\geq&k^2n_{min}q_{min}\left\| H_Qg \right\|^2 -k^2n_{min}q_{min}/2\left\|H_Qg \right\|^{2}=k^2n_{min}q_{min}/2\left\|H_Qg \right\|^{2}
\nonumber\\
&\geq&k^2n_{min}q_{min}/2\left\|H_{Q_0}g \right\|^{2},
\label{6.4}
\end{eqnarray}
and $\mathrm{dim}(H^{*}_QV \cup W) < \infty$. By the same argument in Theorem 1.1, there exists a sequence $(g_m)_{m\in \mathbb{N}} \subset  (H^{*}_QV \cup W)^{\bot}$ such that $\left\|H_{Q_0}w_{m} \right\| \to \infty$ and $\left\|H_{B}g_{m} \right\| \to 0$ as $m \to \infty$, which contradicts (\ref{6.4}). Therefore, we have $\mathrm{Re}N \not\leq_{\mathrm{fin}} \alpha H^{*}_{B}H_{B}$. Theorem 1.2 has been shown. \qed
\vspace{4mm}
\par
By the same argument in Theorem 1.2 we can show the following.
\begin{cor}
Let $B \subset \mathbb{R}^2$ be a bounded open set. Let Assumption hold, and assume that there exists $q_{min}<0$ and $q_{max}<0$ such that $q_{min} \leq q\leq q_{max}$ a.e. in $Q$. Then for $\alpha>k^2n_{max}|q_{min}|$, 
\begin{equation}
Q \subset B \ \ \ \  \Longleftrightarrow \ \ \ \ -\mathrm{Re}N \leq_{\mathrm{fin}} \alpha H^{*}_{B}H_{B},\label{6.5}
\end{equation}
\end{cor}
\section{Numerical examples}
In Section 7, we discuss the numerical examples based on Theorem 1.1. We consider the following two supports $Q_1$ and $Q_2$ of functions $q_1$, $q_2$ (see Figure 1):
\begin{description}
  \item[(1)] $Q_1=\left\{(x_1, x_2) | (x_1-0.5)^2 + (x_2-0.5)^2 < (0.2)^2 \right\}$
  \item[(2)] $Q_2=\left\{(x_1, x_2) | \left((x_1-0.5)/0.15 \right)^2 + \left((x_2-0.6)/0.3 \right)^2 < 1 \right\}$
\end{description}
where $q_1$ and $q_2$ are defined by
\begin{equation}
q_j(x):=\left\{ \begin{array}{ll}
1 & \quad \mbox{for $x \in Q_j $}  \\
0 & \quad \mbox{for $x \notin Q_j$}
\end{array} \right.
\end{equation}

\begin{figure}[htbp]
\begin{tabular}{c}
\begin{minipage}{0.45\hsize}
  \begin{center}
   \includegraphics[scale=0.525]{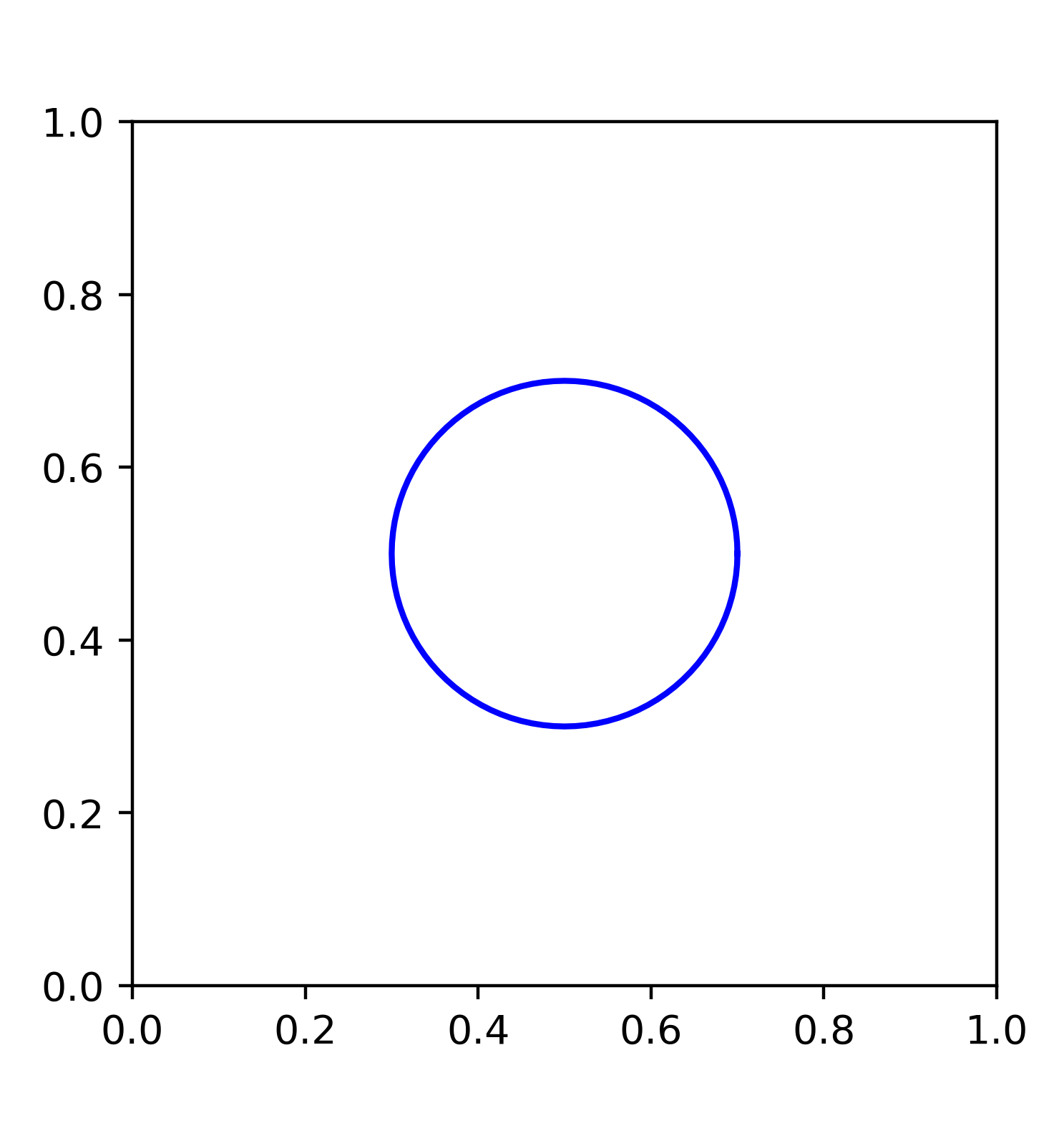}
  \end{center}
  \vspace{-5mm}
  \subcaption*{$Q_1$}
 \end{minipage}
 \begin{minipage}{0.45\hsize}
 \begin{center}
  \includegraphics[scale=0.525]{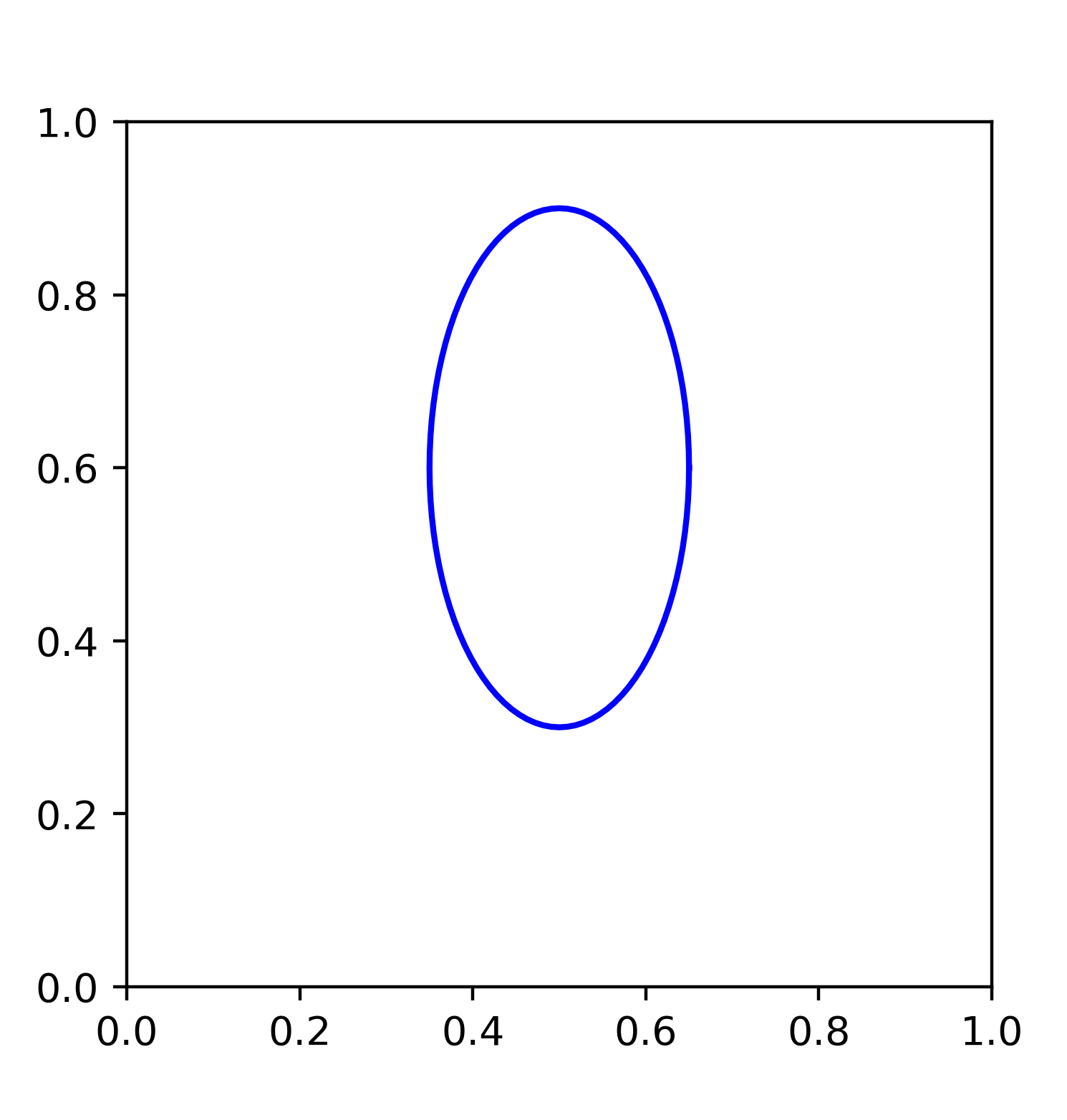}
 \end{center}
 \vspace{-5mm}
  \subcaption*{$Q_2$}
 \end{minipage}
\end{tabular}
\caption{The original shape}
\end{figure}
\vspace{5mm}

Based on Theorem 1.1, the indicator function in our examples is given by
\begin{equation}
I(B):= \# \left\{\mathrm{negative} \ \mathrm{eigenvalues}  \ \mathrm{of} \ \mathrm{Re}N-\alpha H^{*}_{B}H_{B} \right\}
\end{equation}
The idea to recover $Q_j$ is to plot the value of $I(B)$ for many of small $B$ in the certain sampling region. Then, we expect from Theorem 1.1 that the value of the function $I(\sigma)$ is low if $B$ is  included in $Q_j$.
\par
We consider the sampling region by $[0,R]\times[0,R]$ with some $R>0$. The test domain $B$ is given by the small square $B_{i,j}:=z_{i,j}+[-R/2M,R/2M]^2$ where the location $z_{i,j}=(Ri/M, Rj/M)$ $(i,j=1,...,M)$ and $M$ is some large number.
\par
The near field operator $N$ is discretized by the matrix
\begin{equation}
N \approx \frac{b-a}{d} \bigl(u^{s}(x_l, x_p) \bigr)_{1 \leq l,p \leq d} \in \mathbb{C}^{d \times d}
\end{equation}
where $\hat{x}_l=\bigl(a+\frac{(b-a)l}{d}, m \bigr)$, and the scattered field $u^{s}$ of the problem (\ref{1.1})--(\ref{1.2}) is a solution of the following integral equation
\begin{equation}
u^s(x,z)=k^2\int_{Q} q(y)n(y)u^s(y,z)G_n(x,y)dy+k^2\int_{Q} q(y)n(y)\overline{G_n(y,z)}G_n(x,y)dy. \label{7.4}
\end{equation}
In our examples we fix $R=1$, $M=100$, $d=30$, $a=-25$, $b=25$, $m=20$, and $n\equiv 1$. Figure 2 is given by plotting the values of the indicator function 
\begin{equation}
I_{square}(z_{i,j}):=I(B_{i,j}), \ i, j = 1,..., 100,
\end{equation}
for two different supports $Q_1$ and $Q_2$ of true functions $q_1$ and $q_2$, and for two different parameters $\alpha=10, 20$ in the case of wavenumber $k=5$.

\begin{figure}[htbp]
\hspace{-0.5cm}
\begin{tabular}{c}
\begin{minipage}{0.5\hsize}
  \begin{center}
   \includegraphics[scale=0.55]{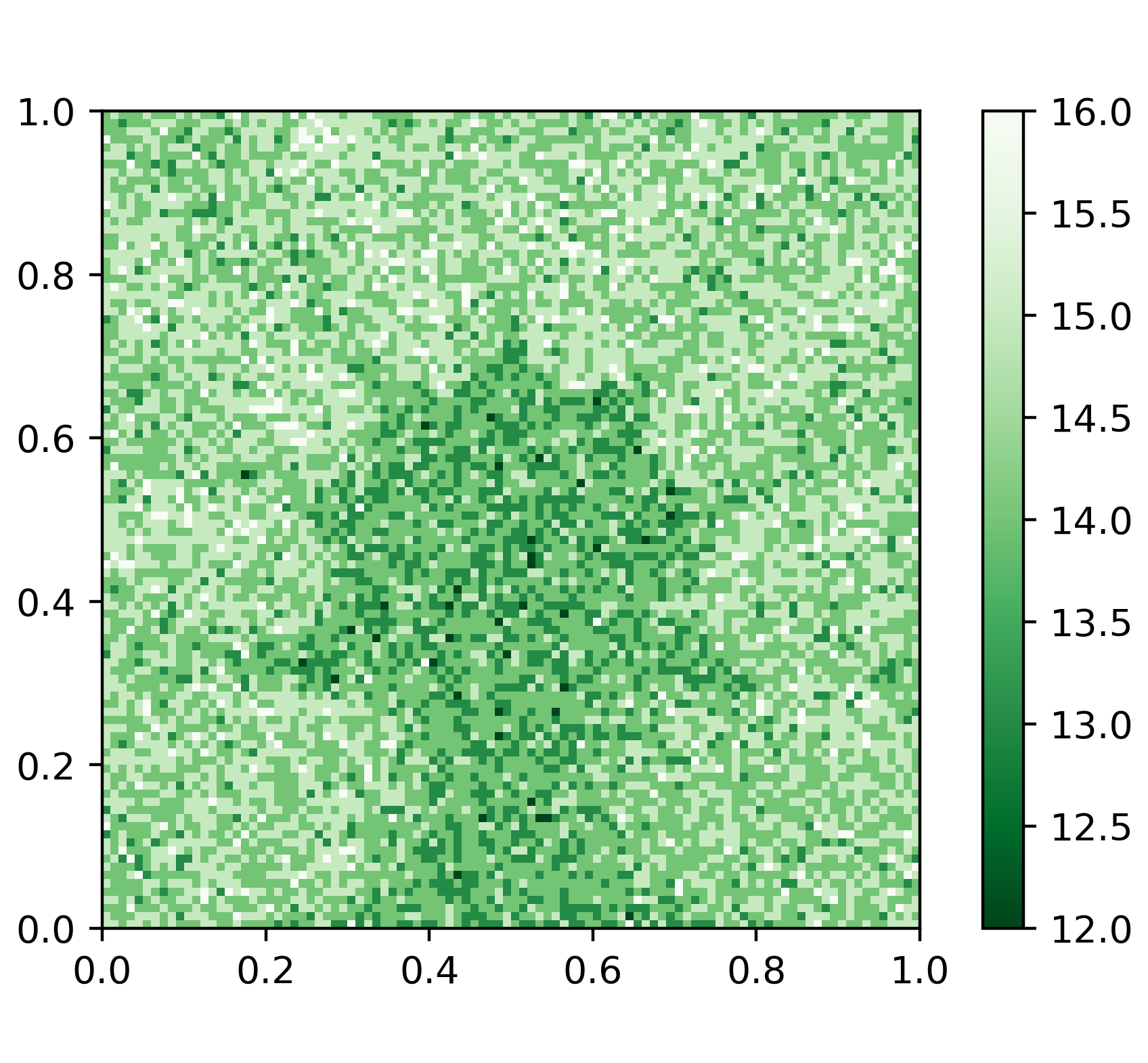}
  \end{center}
  \vspace{-5mm}
  \subcaption*{$Q_1$, $\alpha=10$}
 \end{minipage}
 \begin{minipage}{0.5\hsize}
 \begin{center}
  \includegraphics[scale=0.55]{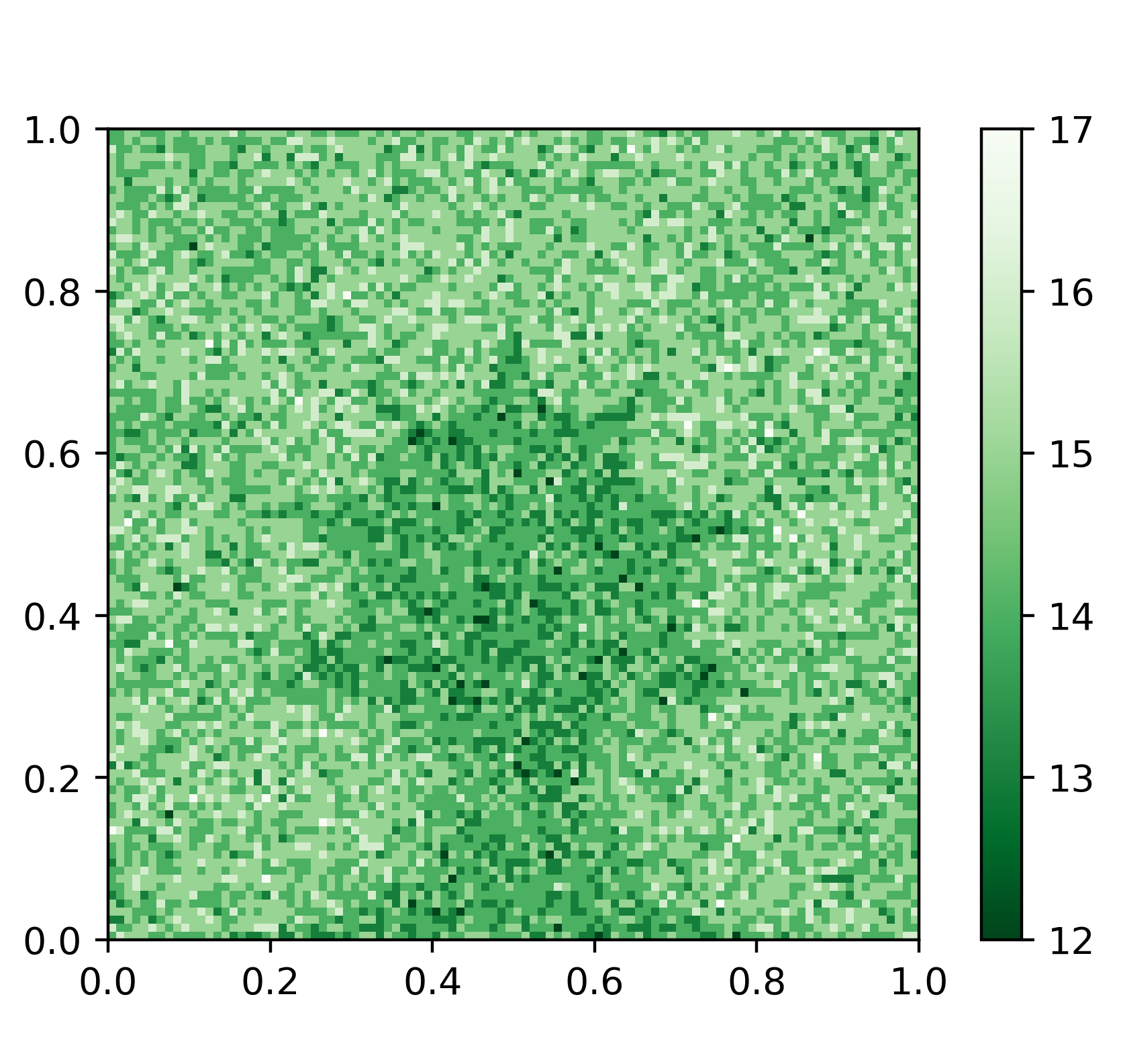}
 \end{center}
  \vspace{-5mm}
  \subcaption*{$Q_1$, $\alpha=20$}
 \end{minipage}
\hspace{2cm}
\end{tabular}

\hspace{-0.5cm}
\begin{tabular}{c}
\begin{minipage}{0.5\hsize}
  \begin{center}
   \includegraphics[scale=0.55]{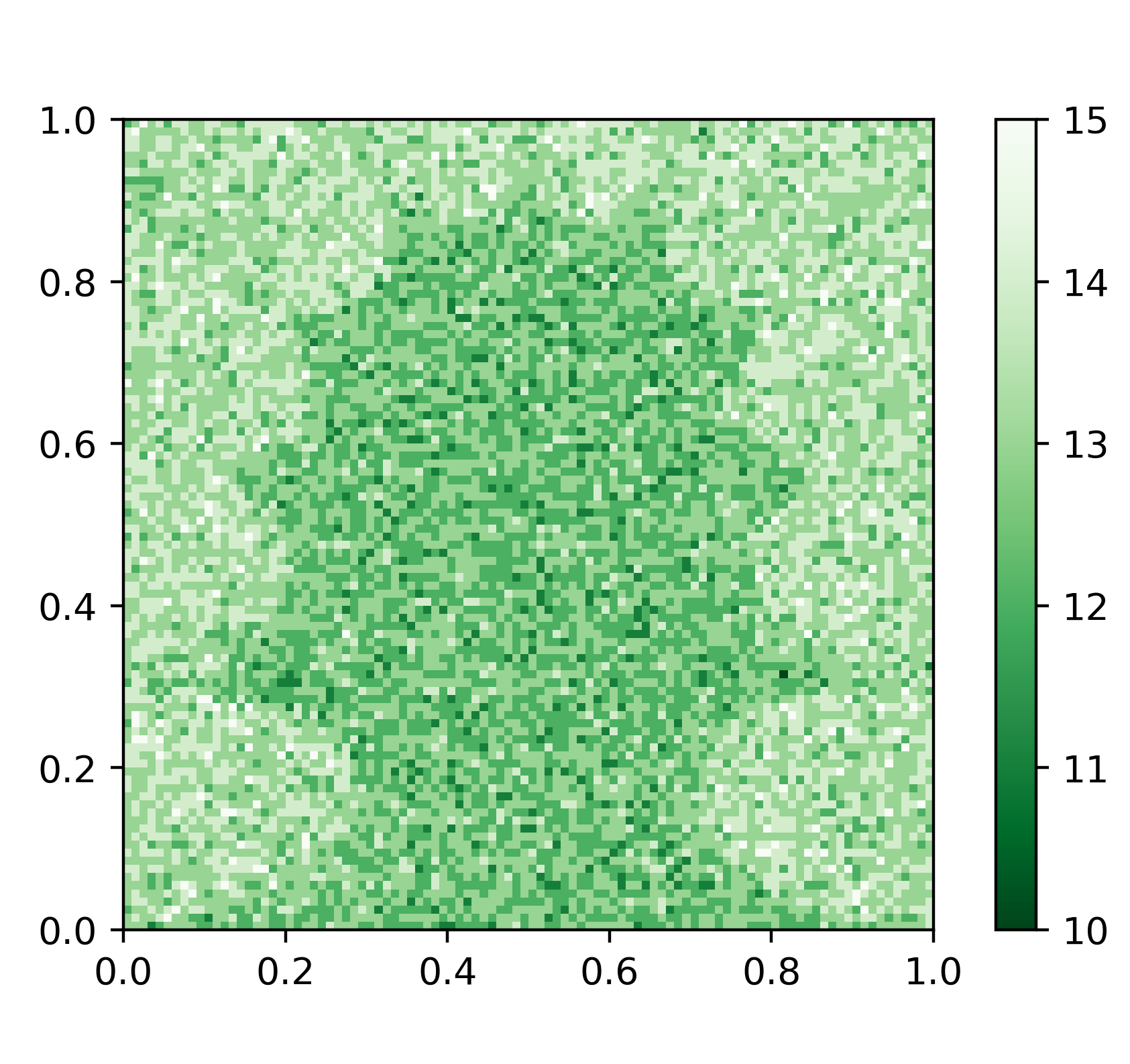}
  \end{center}
  \vspace{-5mm}
  \subcaption*{$Q_2$, $\alpha=10$}
 \end{minipage}
 \begin{minipage}{0.5\hsize}
 \begin{center}
  \includegraphics[scale=0.55]{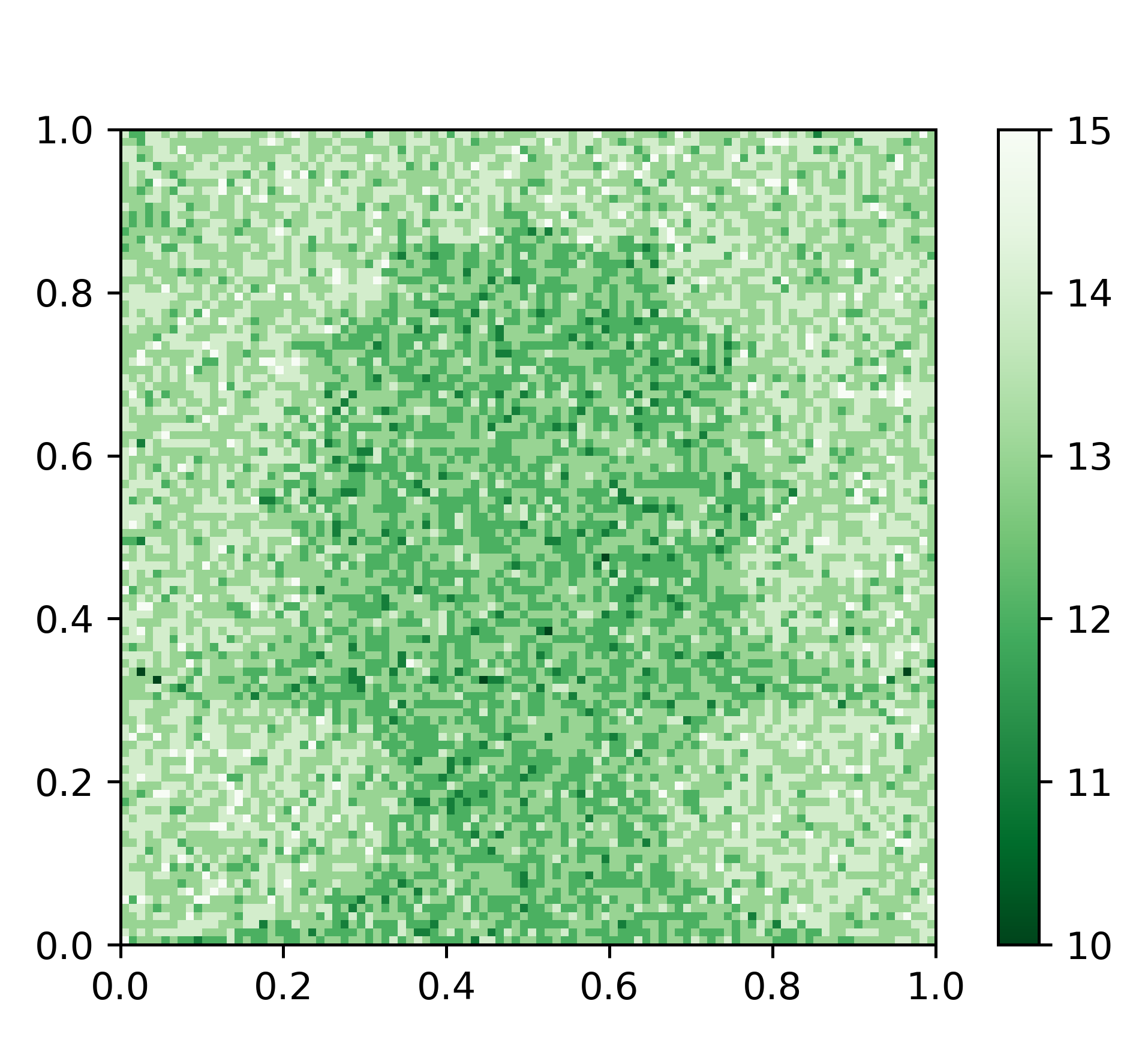}
 \end{center}
  \vspace{-5mm}
  \subcaption*{$Q_2$, $\alpha=20$}
 \end{minipage}
\end{tabular}
\caption{Reconstruction by the indicator function $I_{square}$ in the case of wavenumber $k=5$}
\end{figure}
\section*{Acknowledgments}
The author thanks to Professor Andreas Kirsch, who supports him in this study.


Graduate School of Mathematics, Nagoya University, Furocho, Chikusa-ku, Nagoya, 464-8602, Japan \par
e-mail: takashi.furuya0101@gmail.com


\begin{thebibliography}{14}

\bibitem{Furuya1}T. Furuya,
{\em Scattering by the local perturbation of an open periodic waveguide in the half plane}, preprint arXiv:1906.01180, (2019).


\bibitem{Furuya2}T. Furuya, T. Daimon, R. Saiin,
{\em The monotonicity method for the inverse crack scattering problem}, to appear in Inverse Probl. Sci. Eng., (2020).

\bibitem{R. Griesmaier and B. Harrach}R. Griesmaier, B. Harrach, {\em Monotonicity in inverse medium scattering on unbounded domains}, SIAM J. Appl. Math. {\bf78}, (2018), no. 5, 2533--2557.

\bibitem{B. Harrach and V. Pohjola and M. Salo1}
B. Harrach, V. Pohjola, M. Salo, {\em Dimension bounds in monotonicity methods for the Helmholtz equation}, SIAM J. Appl. Math. {\bf51}, (2019), no. 4, 2995--3019.

\bibitem{B. Harrach and V. Pohjola and M. Salo2}
B. Harrach, V. Pohjola, M. Salo, {\em  Monotonicity and local uniqueness for the Helmholtz equation}, Anal. PDE, {\bf12}, (2019), no. 7, 1741--1771.

\bibitem{B. Harrach and M. Ullrich1}
B. Harrach, M. Ullrich, {\em  Local uniqueness for an inverse boundary value problem with partial data}, Proc. Amer. Math. Soc. {\bf 145}, (2017), no. 3, 1087--1095.

\bibitem{B. Harrach and M. Ullrich2}
B. Harrach, M. Ullrich, {\em Monotonicity based shape reconstruction in electrical impedance tomography}, SIAM J. Math. Anal. {\bf 45}, (2013), no. 6, 3382--3403.

\bibitem{Kirsch}A. Kirsch, {\em The factorization method for a class of inverse elliptic problems}, Math. Nachr. {\bf 278}, (2004), 258--277.

\bibitem{Kirsch and Grinberg}A. Kirsch and N. Grinberg, {\em The factorization method for inverse problems}, Oxford University Press, (2008).

\bibitem{Kirsch and Lechleiter1}
A. Kirsch, A. Lechleiter, {\em A radiation condition arising from the limiting absorption principle for a closed full- or half-waveguide problem}, Math. Methods Appl. Sci. {\bf 41}, (2018), no. 10, 3955--3975.

\bibitem{Kirsch and Lechleiter2}
A. Kirsch, A. Lechleiter, {\em The limiting absorption principle and a radiation condition for the scattering by a periodic layer}, SIAM J. Math. Anal. {\bf 50} (2018), no. 3, 2536--2565.

\bibitem{Lakshtanov and Lechleiter}
E. Lakshtanov, A. Lechleiter, {\em Difference factorizations and monotonicity in inverse medium scattering for contrasts with fixed sign on the boundary}, SIAM J. Math. Anal. {\em 48} (2016), no. 6, 3688--3707.

\bibitem{Lechleiter1}A. Lechleiter,
{\em The factorization method is independent of transmission eigenvalues}, Inverse Probl. Imaging {\bf 3} (2009), 123--138.

\bibitem{Lechleiter2}A. Lechleiter,
{\em The Floquet-Bloch transform and scattering from locally perturbed periodic surfaces}, J. Math. Anal. Appl. {\bf 446}, (2017), no. 1, 605--627.

\bibitem{McLean}W. McLean,
{\em Strongly elliptic systems and boundary integral equations}, Cambridge University Press, Cambridge, (2000).


\end{thebibliography}
\end{document}